\documentclass[11pt]{article}

\usepackage{amsfonts,amsmath,amssymb}
\usepackage{
amsthm}
\usepackage{epsfig}
\usepackage{graphicx}

\newtheorem{theorem}{Theorem}[section]
\newtheorem{proposition}[theorem]{Proposition}
\newtheorem{lemma}[theorem]{Lemma}

\newtheorem{corollary}[theorem]{Corollary}

\newtheorem{observation}[theorem]{Observation}
\newtheorem{definition}[theorem]{Definition}

\newtheorem{conjecture}[theorem]{Conjecture}

\newcommand{\Imin}[0]{I_{{\rm min}}}

\newcommand{\beq}[1]{\begin{equation}\label{#1}}
\newcommand{\enq}[0]{\end{equation}}

\newcommand{\mn}[0]{\medskip\noindent}
\newcommand{\nin}[0]{\noindent}

\newcommand{\sub}[0]{\subseteq}
\newcommand{\sm}[0]{\setminus}
\newcommand{\ra}[0]{\rightarrow}

\newcommand{\A}[0]{{\cal A}}
\newcommand{\B}[0]{{\cal B}}
\newcommand{\f}[0]{{\cal F}}
\newcommand{\C}[0]{{\cal C}}
\newcommand{\I}[0]{{\cal I}}
\newcommand{\J}[0]{{\cal J}}

\newcommand{\ga}[0]{\alpha }
\newcommand{\gl}[0]{\lambda }
\newcommand{\gO}[0]{\Omega}

\newcommand{\gc}[0]{\gamma}

\newcommand{\0}[0]{\emptyset}

\newcommand{\Cor}[0]{\mathrm{Cor}}

\usepackage[usenames]{color}

\begin{document}

\title{Chv\'{a}tal's Conjecture and Correlation Inequalities}

\author{
Ehud Friedgut\thanks{Faculty of Mathematics and Computer Science, The Weizmann Institute of Science, Rehovot, Israel.
Research supported in part by ISF grant 0398246, and Minerva grant 712023. E-mail:~ehudf@math.huji.ac.il}, 
Jeff Kahn\thanks{Department of Mathematics, Rutgers University, Piscataway NJ 08854 USA.
Partially supported by NSF grants DMS1201337 and DMS1501962, and BSF grant 2014290.
E-mail:~jkahn@math.rutgers.edu}, Gil Kalai\thanks{Einstein Institute of Mathematics, Hebrew University, Jerusalem, Israel.
Partially supported by ERC advanced grant 320924, NSF grant DMS1300120, and BSF grant 2014290.
E-mail:~kalai@math.huji.ac.il}, and
Nathan Keller\thanks{Department of Mathematics, Bar Ilan University, Ramat Gan, Israel. Partially supported by ISF grant 402/13, BSF grant 2014290, and by the Alon Fellowship. E-mail:~nathan.keller27@gmail.com}  \\
}

\maketitle

\begin{abstract}
Chv\'{a}tal's conjecture in extremal combinatorics asserts that for any decreasing
family $\mathcal{F}$ of subsets of a finite set $S$, there is a largest intersecting
subfamily of $\mathcal{F}$ consisting of all members of $\mathcal{F}$
that include a particular $x \in S$. In this paper we reformulate the conjecture in
terms of influences of variables on
Boolean functions and correlation inequalities, and study
special cases and variants 
using tools from discrete Fourier
analysis.
\end{abstract}

\section{Introduction}\label{Intro}

\nin {\em Definitions.}
A family $\cal G$ of subsets of $[n]=\{1,2,\ldots,n\}$ is \emph{intersecting} if
$A \cap B \neq \emptyset$ for
any $A,B \in \cal G$, and \emph{increasing} if $A\supset B\in\cal G$
implies $A\in \cal G$ (and similarly for \emph{decreasing}).

\medskip
One of the seminal results (maybe \emph{the} seminal result)
of
extremal combinatorics is the Erd\H{o}s-Ko-Rado
theorem~\cite{EKR}, which says that, for $k\leq n/2$,
the maximum size of an intersecting subfamily of the family $\f$ of all $k$-subsets
of $[n]$ is $ {{n-1}\choose{k-1}}$, the number of $k$-sets containing some fixed
$x \in [n]$.
Given this, it is natural to ask
whether something similar holds for other $\f$'s.
A celebrated 1972 conjecture of Chv\'{a}tal \cite{Chvatal}
says that this is true for every {\em decreasing} $\f$:

\begin{conjecture}[Chv\'atal's Conjecture]\label{Conj:Chvatal}
For any decreasing
$\mathcal{F}\sub 2^{[n]}$, some largest intersecting
subfamily has the form $\{A\in \f:x\in A\}$.
\end{conjecture}
\nin
Of course it is no longer the case that {\em any} $x$ will suffice, and the difficulty of
identifying a suitable $x$ is a central reason for the conjecture's intractability.
Chv\'{a}tal's Conjecture has been the subject of many papers\footnote{A list of more than 20 papers
directly related to the conjecture appears at the website:
http://users.encs.concordia.ca/chvatal/conjecture.html.} (and surely far more effort
than this published record indicates),
but progress to date has been
limited, dealing mostly with either very special cases or variants.


In this paper we suggest an analytic approach. We show that Chv\'atal's Conjecture
can be restated in terms of influences (defined below) and correlation inequalities,
providing an opening for use of tools from discrete
Fourier analysis.\footnote{
Sean Eberhard independently considered similar relations, motivating
a MathOverflow question~\cite{Eberhard}.}
Our approach grew out of discussions
of a stronger form of Chv\'{a}tal's conjecture suggested by the second author about 25 years
ago (see Section~\ref{sec:Kahn}).

We first recall a few definitions.
In what follows, we identify subsets of $[n]$ with elements of the discrete
cube $\Omega = \{0,1\}^n$ in the natural way and write $\mu$ for
uniform
measure on $\gO$. The \emph{correlation} between $\A,\B\sub \Omega $ is
$\mathrm{Cor}(\A,\B) = \mu(\A \cap \B) - \mu(\A)\mu(\B)$.
More generally for $f,g:\gO\ra\mathbb{R}$ we use
$\mathrm{Cor}(f,g) = \mathbb{E}_\mu[fg] - \mathbb{E}_\mu[f] \mathbb{E}_\mu[g]$
(so
$\mathrm{Cor}(\chi_{_\A},\chi_{_\B}) =\mathrm{Cor}(\A,\B)$,
where we use $\chi$ for indicator;
of course $\mathrm{Cor}(f,g)$ is just the covariance of $f$ and $g$).
A family $\f$ is said to be
\emph{antipodal}
if $|\f\cap \{A,A^c\}|=1$ for each $A\sub [n]$ (with $A^c$ the complement of $A$).


The \emph{influence} of the $k^{th}$ variable on $\A \sub \Omega $ is
\[
I_k(\A) = 2\mu(\{x \in \A | x \oplus e_k \not \in \A\}),
\]
where $x \oplus e_k$ is gotten from $x$ by replacing $x_k$
by $1-x_k$. The \emph{total influence} of $\A$ is $I(\A) = \sum_{k=1}^n I_k(\A)$
and we write $\Imin(\A)$ for $\min_{1 \leq k \leq n} I_k(\A).$

\medskip
Recall that Harris' seminal correlation inequality \cite{Harris} says
that $\mathrm{Cor}(\A,\B)\geq 0$,
for increasing $\A,\B$. Michel Talagrand~\cite{Talagrand96}  initiated the study of: ``How much are increasing sets positively correlated?'', and
this question will be
a central theme for us as well. As we will see, Chv\'{a}tal's Conjecture
can also be formulated as a correlation inequality, {\em viz.}

\begin{conjecture}\label{Conj:Equiv-Chvatal}
For any increasing $\A $ and increasing antipodal
$\B$ (both $ \sub \Omega $),
\beq{basic}
\mathrm{Cor}(\A,\B) \geq \tfrac14 \Imin(\A).
\enq
\end{conjecture}
\nin
The equivalence is shown in Section \ref{sec:equivalent}.
We will also be interested in a weaker but more general possibility:

\begin{conjecture}
\label{Conjecture:Balanced}

For any increasing $\A ,\B\sub \Omega $
\begin{equation}\label{Eq:Question-Balanced}
\mbox{$\mathrm{Cor}(\A,\B) \geq c \Imin (\A) \mu (B) (1-\mu(B)).$}
\end{equation}
for some fixed (positive) $c$.
\end {conjecture}

\nin
As we will see in Section~\ref{sec:sub:most-general},
(\ref {Eq:Question-Balanced})
is not true
with $c > \ln 2$,
even if $\B$ is balanced (i.e., $\mu(\B)=1/2$); in particular,
the antipodality in Conjecture~\ref{Conj:Equiv-Chvatal} cannot be replaced
by the weaker assumption that $\B$ is balanced.
On the other hand, Kahn's strong from of Chv\'atal's conjecture (Conjecture~\ref{Conj:Kahn} below)
implies that (\ref {Eq:Question-Balanced}) does hold with $c=1/2$,
and the possibility  that this relaxation loses
{\em only} a constant factor seems to us one of the more interesting
aspects of the present discussion.


Lower bounds on the correlations of increasing families in terms of influences
were obtained by Talagrand~\cite{Talagrand96}
(as already mentioned; see
Theorem~\ref{TTalagrand} below) and by Keller, Mossel, and Sen~\cite{KMS14}
(Theorem~\ref{TKMS}). In Section \ref{sec:correlation}, we combine these results with results about
influences of an individual family (due to Kahn, Kalai, and Linial \cite{KKL}, and
Talagrand \cite{Talagrand94}), to
prove Conjecture~\ref{Conjecture:Balanced} under some (fairly strong) additional
hypotheses. We also prove, for general increasing families $\A,\B$,

\begin{equation}\label{Eq:KMS+Tala}
\mbox{$\mathrm{Cor}(\A,\B) \geq c \Imin (\A)\mu (B) (1-\mu (B))/ \sqrt {\log (1/\Imin (\A))} $}.
\end{equation}
\nin
These results may be thought of as illustrating connections with existing
Fourier technology. Inequality (\ref {Eq:KMS+Tala}), while weak compared to what we are after,
may serve as a useful benchmark for future research.
In Section~\ref{sec:average} we rely on \cite {Keller09} and, perhaps surprisingly, show that
Conjecture~\ref{Conj:Equiv-Chvatal} is true in some average sense.


In Sections \ref{sec:correlation2}  and \ref{sec:proofs} we propose and study
strengthenings of Harris' inequality
that would imply Conjectures \ref{Conj:Equiv-Chvatal} and \ref{Conjecture:Balanced}.
One possibility is the following
consequence of Kahn's conjecture. (Here we use $\hat \C(S)$ for the Fourier coefficient
$\hat\chi_{_\C}(S)$; Fourier definitions are recalled in the next section.)




\begin{conjecture}\label{conjecture:Kahn-intro}
For any increasing $\A, \B \sub \Omega $,
\begin {equation}
\mathrm{Cor}(\A,\B) \geq c \sum_{k=1}^n I_k(\A)
\sum\{\frac {1} {|S|}\hat \B (S)^2:S\ni k\},
\end {equation}
for some universal $c$.  If $\B$ is antipodal this is true with $c=1$.
\end{conjecture}
\nin
Notice that this
gives a lower bound for $\mathrm{Cor}(\A,\B)$ in terms of a weighted sum of
the influences of $\A$, with the sum of weights in the antipodal case
equal to 1/4.





\section{Reformulation and preliminaries}
\label{sec:equivalent}

This section gives the easy equivalence of the two versions of Chv\'{a}tal's
Conjecture stated in the introduction and some additional background and comments.

\subsection {Reformulation}

\begin{proposition}\label{Prop:Equivalence}
For a decreasing $\mathcal{F} \sub \Omega $ the following
statements are equivalent.

\mn
{\rm (a)}
There is a $k\in [n]$ for which
$\max\{|\B|:\B\sub \mathcal{F}~\mbox{intersecting}\}= |\{A\in \mathcal{F}:k\in A\}|$.

\mn
{\rm (b)}
There is a $k\in [n]$ such that
the maximum correlation of $\f$ with a maximal intersecting $\B$
is attained by $\B=\{A\in \Omega: k\in A\}$.

\mn
{\rm (c)}
For any increasing, antipodal $\B  \sub \Omega $,
$
\mathrm{Cor}(\mathcal{F},\B ) \leq -\frac14 (\Imin(\mathcal{F})).
$
\end{proposition}
\nin
Of course (a) is Conjecture~\ref{Conj:Chvatal}, while
(c) is the same as
Conjecture~\ref{Conj:Equiv-Chvatal}, since, for any
$\A,\B $,
$\mathrm{Cor}(\A,\B )=-\mathrm{Cor}(\A^c,\B )$
(more generally,
$\Cor(1-f,g)=\Cor(1,g)-\Cor(f,g)=-\Cor(f,g)$ for any $f,g$).
%

\begin{proof}
It is obvious (and standard) that the maximum in (a)
is the same as
\[
\max\{|\f\cap \B|: \B ~\mbox{maximal intersecting}\},
\]
and that each maximal intersecting $\B$ has measure 1/2.
(It is easy to see---and was observed e.g. in \cite{EK74}---that $\mathcal{F} $
is maximal intersecting if and only if it is increasing and antipodal.)
Thus for maximal intersecting $\B$ we have
\beq{Bcapf}
|\f \cap \mathcal{B}| = 2^n (\mathrm{Cor}(\f,\B) +
\mu(\f)/2),
\enq
which implies the equivalence of (a) and (b)
(since we maximize the left side of \eqref{Bcapf} by maximizing $\Cor(\f,\B)$).
For the equivalence of (c) we just observe that for $\B$ as in (b)
(sometimes called a ``dictatorship") we have
\[
\Cor(\f,\B ) = -  \mu(\{A\in \f:A\cup\{x\}\not\in \f\})/2
= - I_k(\f)/4.
\]
\end{proof}

\subsection {Harper and Fourier-Walsh}

Harper's classic edge-isoperimetric inequality~\cite{Harper} says
(though not originally in this language) that for all $\A  \sub \Omega$,
\beq{THarper}
I(\A ) \geq 2 \mu(\A ) \log_2 (1/\mu(\A )).
\enq
In particular $I(\A) \ge 1$ for balanced $\A$.

\mn
{\em Definition.}
For $f:\Omega  \rightarrow \mathbb{R}$, the {\em Fourier-Walsh
expansion} of $f$ is the (unique) representation
\[
f = \sum\{\alpha_S u_S: S\sub [n]\},
\]
where $
u_S(T)=(-1)^{|S \cap T|}
$
for $T \sub [n]$.
The ({\em Fourier}) coefficients $\alpha_S$ are also denoted $\hat f(S)$.

\medskip
Since
$\{u_S\}$ is an orthonormal basis for the
space of functions $f:\Omega  \rightarrow \mathbb{R}$ (relative to
the usual inner product $\langle\cdot,\cdot\rangle$
with respect to uniform
measure), the representation is indeed unique, with
$\hat f(S) = \langle f ,u_S\rangle$,
and we have {\em Parseval's identity:}
\beq{Parseval}
\langle f,g\rangle=\sum \hat{f}(S)\hat{g}(S)~~~ \forall f,g.
\enq
Thus (since $\mu(f)=\hat f(\0)$),
\beq{Parseval2}
\mathrm{Cor}(f,g) = \sum_{S \ne \emptyset} \hat f(S)\hat g(S).
\enq
As we have already done above, we will sometimes use $\hat \C (S)$ for $\hat{\chi_{_\C} }(S)$.
It is standard (see e.g. \cite{KKL}) that for any $\A  \sub \Omega $ and $i\in [n]$,
$I_i(\A ) = 4\sum\{\hat \A (S)^2:S\ni i\}$. If $\A$ is decreasing then also $I_i(\A)=2\hat \A(\{i\})$
and if $\A$ is increasing then  $I_i(\A)=-2\hat \A(\{i\})$.

\subsection {The dream relation}

The following observation is simple but crucial for our line of thought
({\em cf.} the aforementioned Theorems~\ref{TTalagrand} and~\ref{TKMS}).

\begin{proposition}\label{Prop:basic}
Let $\A,\B$ be increasing events with $\mu(\B)=t$.  If
\beq{Eq:dream}
\mbox{$\mathrm{Cor}(\A ,\B ) \geq  \frac {1}{4} \sum I_k(\A )I_k(\B ),$}
\enq
then
\beq{Eq:dream-cons}
\mbox{$\mathrm{Cor}(\A ,\B ) \geq \frac {1}{2} \Imin(\A) t \log_2 (1/t).$}
\enq
In particular, if $\B$ is balanced then $\mathrm{Cor}(\A ,\B ) \geq  \frac {1}{4} I_{min}(\A )$.
\end {proposition}
\begin {proof} Combining \eqref{Eq:dream} and Harper's inequality gives
\[
\mathrm{Cor}(\A ,\B ) \geq  \frac {1}{4} \sum I_k(\A )I_k(\B ) \ge
\frac {1}{4} \Imin(\A)I(\B)\ge \frac {1}{2} \Imin (A)t \log_2(1/t).
\]
\end {proof}

\nin
Again, Proposition~\ref{Prop:basic} is mainly motivational;
neither \eqref{Eq:dream} (the ``dream relation'')
nor its consequence \eqref{Eq:dream-cons} is true in general.
In this paper, we consider weaker statements of similar type.




\section{The conjectures of Kleitman and Kahn}
\label{sec:Kahn}

In this section we describe an earlier analytic approach to Chv\'{a}tal's Conjecture
suggested by the second author in an unpublished manuscript in the early 90s~\cite{Kahn-unpublished}.
This built on a strengthening of Chv\'{a}tal's Conjecture proposed by
Kleitman~\cite{Kleitman} in 1979.


\begin{definition}\label{flow}
Let $f,g:\Omega  \rightarrow \mathbb{R}^{+}$. We say that $f$ \emph{flows to} $g$
if there exists $v:\Omega  \times \Omega  \rightarrow \mathbb{R}^+$ such that:
\begin{enumerate}
\item For any $A \in \Omega $, we have $\sum_B v(A,B) = f(A)$.

\item For any $B \in \Omega $, we have $\sum_A v(A,B) = g(B)$.

\item If $A \not \sub B$, then $v(A,B)=0$.
\end{enumerate}
\end{definition}
\nin
Equivalently ({\em via} max-flow min-cut),
$f$ flows to $g$ if $\sum_A f(A) = \sum_A g(A)$, and
$f(\f)\geq g(\f)$ for every decreasing family $\f$
(where $f(\f) =
\sum_{A \in \mathcal{F}} f(A)$).

\mn
{\em Notation.}
For a ``principal'' family $\f=\f_i = \{A : i \in A\}$,
we set $\chi_{_\f}=\chi_{_i}$
(recalling that $\chi_{_\f}$ is
the indicator of $\f$).

\medskip
The following strengthening of Chv\'{a}tal's
conjecture was proposed by Kleitman \cite{Kleitman}.
\begin{conjecture}\label{CKleitman}
For any maximal (w.r.t. inclusion) intersecting
$\f \sub \Omega $,
there is a convex combination
$\sum_{i=1}^n \lambda_i \chi_{_i}$
of $\chi_1,\ldots,\chi_n$
that flows to $\chi_{_\f}$.
\end{conjecture}
\nin
Fishburn~\cite{Fishburn} observed that this is equivalent to a
``functional" form of Chv\'{a}tal's conjecture, {\em viz.}
\begin{conjecture}\label{Conj:Eq-Kleitman}
For any nonincreasing $g:\Omega  \rightarrow \mathbb{R}^+$,
$g(\f)$ is maximized over intersecting families $\f$ by some $\f_i$.
\end{conjecture}
\nin
Of course Chv\'{a}tal's Conjecture is just
Conjecture~\ref{Conj:Eq-Kleitman} for $\{0,1\}$-valued $g$.

\medskip

The suggestion of~\cite{Kahn-unpublished} is a particular set of
$\lambda_i$'s for Kleitman's conjecture; these are most easily described in terms of
the Fourier-Walsh coefficients.

\medskip
For $f:\Omega  \rightarrow \mathbb{R}$, set
$f^*(x) = \max(f(x),0)^2$ (thus $f^*(x)=f(x)^2$ if $f(x)$ is nonnegative and $f^*(x)=0$
otherwise).
We call $f$ \emph{antipodal} if $f(A^c)=-f(A)$
for any $A \sub [n]$. In particular, if $\mathcal{F}$ is an antipodal family, then
$f = 2 \cdot 1_{\mathcal{F}}-1$ is an antipodal function.

\begin{conjecture}[\cite{Kahn-unpublished}]\label{Conj:Kahn}
For any nondecreasing, antipodal $f:\Omega  \rightarrow \mathbb{R}$, if
\[
\lambda_i = \lambda_i(f) = \sum\{\hat f(S)^2:\max(S) = i \}
~~~~~1 \leq i \leq n,
\]
then $\tilde{f} = \sum_{i=1}^n \lambda_i \chi_{_i}$ flows to $f^*$.
\end{conjecture}
\nin
Note that for an antipodal $f$, $\hat{f}(\0)=0$, so \eqref{Parseval} gives
\beq{convex}
\mbox{$\sum \gl_i(f) =\sum_{S\neq\0}\hat{f}^2(S)
= \langle f,f\rangle -\hat{f}^2(\0)
=2^{-n}\sum f^2(T) = 2^{-(n-1)}\sum f^*(T).$}
\enq
Thus $\sum \tilde{f}(T) = 2^{n-1}\sum \gl_i =\sum f^*(T)$,
a prerequisite for the conclusion of Conjecture~\ref{Conj:Kahn}.
In particular, when $f = 2 \cdot 1_{\mathcal{F}}-1$ with $\f$ maximal intersecting,
the $\gl_i$'s are convex coefficients, and
in this case
Conjecture~\ref{Conj:Kahn} strengthens
Kleitman's Conjecture~\ref{CKleitman} by specifying the $\gl_i$'s.
%
As noted following Definition~\ref{flow}, Conjecture~\ref{Conj:Kahn} is equivalent to
\begin{conjecture}\label{Conj:Kahn-Equivalent}
If $f:\Omega  \rightarrow \mathbb{R}$ is nondecreasing, antipodal and
$\mathcal{I}\sub \Omega $ is decreasing, then (with $\tilde{f}$ as above)
\beq{AinmathcalI}
\sum_{A \in \mathcal{I}} \tilde{f}(A) \geq \sum_{A \in \mathcal{I}} f^*(A).
\enq
\end{conjecture}
\nin
As noted in~\cite{Kahn-unpublished}, the following, superficially more general, version of
Conjecture~\ref{Conj:Kahn} is again equivalent.
\begin{conjecture}\label{Conj:Kahn-Convex}
Let $f:\Omega  \rightarrow \mathbb{R}$ be non-decreasing and antipodal.
For each $S \sub [n]$, let $\lambda_S: [n] \rightarrow \mathbb{R}^+$ be some
function satisfying
\[
\mbox{$\sum_{i=1}^n \lambda_S(i) = \hat f(S)^2 ~~$ and $ ~~
\lambda_S(i)=0 ~\forall i \not \in S$,}
\]
\mn
and set $\lambda_i = \sum_{i \in S} \lambda_S(i)$. Then $\sum_{i=1}^n \lambda_i \chi_{_i}$
flows to $f^*$.
\end{conjecture}
\nin
Conjecture~\ref{Conj:Kahn} is the special case gotten by setting
$\lambda_S(i) = \hat f(S)^2 \chi_{\{i = \max S\}}$.
Another natural choice is
\beq{average}
\lambda_S(i) = |S|^{-1}\hat f(S)^2 \chi_{\{i \in S\}}.
\enq
Conjecture~\ref{Conj:Kahn-Convex} with this choice is weaker than Conjecture~\ref{Conj:Kahn},
but of course still sufficient for Conjecture~\ref{CKleitman}.

As observed in~\cite{Kahn-unpublished},
Conjecture~\ref{Conj:Kahn} (or Conjecture~\ref{Conj:Kahn-Equivalent})
also implies a natural extension of
Chv\'{a}tal's Conjecture
to general (not necessarily maximal) increasing, intersecting families:
\begin{conjecture}\label{Conj:Kahn-small}
For any increasing, intersecting $\mathcal{F} \sub \Omega $ and
decreasing $\mathcal{I} \sub \Omega $, there is an $i$ such that
\[
\frac{|\mathcal{F}_i \cap \mathcal{I}|}{2^{n-1}} \geq \frac{|\mathcal{F} \cap
\mathcal{I}|}{|\mathcal{F}|}.
\]
\end{conjecture}

\nin
Some further discussion of
Conjecture~\ref{Conj:Kahn-Equivalent} and variants, and in particular, of some surprising
cases in which the conjecture is tight,
is provided in~\cite{Kahn-unpublished}.
Correlation reformulations of
the conjectures of \cite{Kahn-unpublished}
are given in Section~\ref{sec:correlation}.

\section{Chv\'{a}tal's Conjecture and off-the-shelf correlation inequalities}
\label{sec:correlation}

We have already mentioned the fundamental
inequality of Harris~\cite{Harris}, asserting positive (i.e., nonnegative) correlation
of any two increasing subsets of $\Omega $.
(There are also some well-known extensions, in particular the
``FKG Inequality" of~\cite{FKG} and the ``Four Functions Theorem" of \cite{Ahlswede-Daykin}.)
In 1996, Talagrand~\cite{Talagrand96} proved a lower bound on
the correlation in Harris' Inequality in terms of influences. In 2012, Keller, Mossel, and Sen~\cite{KMS14}
proved an alternative lower bound (incomparable with Talagrand's). As
Conjecture~\ref{Conj:Equiv-Chvatal} (our reformulation of Chv\'{a}tal's Conjecture) again
asks for a lower
bound on correlation of increasing
families in terms of influences, it is natural to hope that lower bounds along the lines
of~\cite{KMS14,Talagrand96}
may help in proving it.
Here we review the above bounds and see what they have to say about
Conjecture~\ref{Conj:Equiv-Chvatal}.

We assume from now on (as we clearly may) that $\Imin (\A)$ is positive.
Note that in what follows ``Chv\'atal's Conjecture" usually refers to the form in
Conjecture~\ref{Conj:Equiv-Chvatal}.

\subsection{Talagrand's inequality}
\label{sec:sub:Talagrand}

In~\cite{Talagrand96}, Talagrand proved the following correlation inequality.
\begin{theorem} 
\label{TTalagrand}
For any increasing $\A ,\B  \sub \Omega $,
\begin{equation}
\mbox{$\mathrm{Cor}(\A ,\B ) \geq c \varphi \left(\sum I_k(\A ) I_k(\B ) \right),$}
\label{Eq:Talagrand}
\end{equation}
where $\varphi(x)=x/\log(e/x)$, and $c$ is a
universal constant.
\label{Thm:Talagrand}
\end{theorem}
\nin
(Here and below sums indexed by $k$ run over $k\in [n]$.)

Combined with Harper's inequality \eqref{THarper},
Theorem~\ref{Thm:Talagrand} yields a weak version of Chv\'{a}tal's Conjecture:
\begin{corollary}\label{Claim:Weak}
For $\A , \B \sub \Omega $ with $\A$ increasing  and $\B$ increasing and antipodal,
\[
\mathrm{Cor}(\A ,\B ) \geq c \varphi(\Imin(\A )),
\]
where $\varphi(x)$ is as in Theorem~\ref{TTalagrand} and $c$ is a universal constant.
\end{corollary}

\begin{proof}
From \eqref{THarper} we have
\[
\mbox{$\sum  I_k(\A ) I_k(\B ) \geq \Imin(\A )
\sum I_k(\B ) \geq \Imin(\A ),$}
\]
which, since $\varphi$ is increasing,
gives Corollary~\ref{Claim:Weak} via
Theorem~\ref{Thm:Talagrand}.
\end{proof}
Let us stress that Corollary~\ref{Claim:Weak} is
{\em much} weaker than Chv\'{a}tal's Conjecture, since $\Imin(\A)$ is always $O(n^{-1/2})$
(it is largest when $\A$ is
``majority"), and is often much smaller.
The following proposition says we can do better if we impose some (restrictive but not unnatural)
assumptions; here we need to recall the ``KKL Theorem" of \cite{KKL}:
\begin{theorem}\label{TKKL}
There is a fixed $c>0$ such that for any $\A  \sub \Omega $, there is a $k \in [n]$ with
\[
I_k(\A ) \geq c \mu(\A )(1-\mu(\A ))  (\log_2 n)/n.
\]
\end{theorem}
\mn
{\em Definition.}
 $\A \sub \Omega $ is \emph{regular}
if
$I_i(\A )=I_j(\A )$ $~\forall ~i,j$ (for an increasing $\A$, this means that the sets
$\{S \in \A: i \in S\}$ are all of the same size).
Of particular interest here are the {\em weakly symmetric}
families, those invariant under transitive subgroups of $\mathfrak{S}_n$.


\begin{proposition}\label{Prop:Weakly-Symmetric}
For each $a>0$ there is a $c=c(a)>0$ such that if
$\A  \sub \Omega $ is increasing with $\mu(\A ) \in ( n^{-a},1- n^{-a})$ and
$\B\sub \gO$ is increasing, balanced and regular, then
\beq{EqWS}
\mathrm{Cor}(\A ,\B ) > c\Imin(\A ).
\enq
\end{proposition}

\nin

\nin
Note that the assumption that $\B$ is regular 
holds in the examples of
Section~\ref{sec:sub:most-general} that give the strongest constraint we know on the $c$
in Conjecture \ref{Conjecture:Balanced}.

\begin{proof}
(We use $c',c''\ldots$ for positive constants depending on $a$.)
The assertion is the same for $\A^c$ as for $\A$
(since $\mathrm{Cor}(\A ,\B )=\mathrm{Cor}(\A ^c,\B^c )$ and
complementation doesn't affect influences), so we may assume $\mu(\A)\leq 1/2$.
Theorem~\ref{TKKL} and our assumptions on
$\B $ give $I_k(\B)>  c' \log n / n$ $\forall k$, implying
\begin{equation}\label{Eq:Prop-Weakly-Symmetric1}
\mbox{$\sum I_k(\A ) I_k(\B ) \geq c' (\log n/n)\sum I_k(\A)
\geq c'\log n\cdot \Imin(\A).$}
\end{equation}
On the other hand, since $\mu(\A)\in ( n^{-a},1/2]$, \eqref{THarper} gives
\[
c' (\log n/n) I(\A ) \geq 2c' (\log n/n) \mu(\A ) \log_2(1/\mu(\A )) > c''a n^{-(a+1)}\log^2 n ,
\]
whence
\begin{equation}\label{Eq:Prop-Weakly-Symmetric2}
\mbox{$\log (e/\sum I_k(\A ) I_k(\B )) < c''' \log n.$}
\end{equation}
From \eqref{Eq:Prop-Weakly-Symmetric1} and
\eqref{Eq:Prop-Weakly-Symmetric2} we have $\varphi(\sum I_k(\A ) I_k(\B ) ) > c\Imin(\A)$,
so \eqref{EqWS} is given by Theorem~\ref{Thm:Talagrand}.
\end{proof}

\nin
{\em Remarks.}
1.  Of course the above proof supports replacement of $\Imin(\A)$ in
\eqref{EqWS} by
the average, $I(\A)/n$, of the $I_k(\A)$'s.  As pointed out
to us by Alex Samorodnitsky, when $\B $ is ``majority"
(the ``fully symmetric" case),
$\mathrm{Cor}(\A ,\B )\geq I(\A)/(4n)$
for {\em any} increasing $\A$; this follows from the fact that
$\A$ contains at least as many sets of size $k$ as of size $n-k$ for any $k>n/2$,
and is exact when $\A$ is $\{\underline{1}\}$ or $\gO\sm \{\underline{0}\}$.

\mn

2.
As shown in \cite{Talagrand96}, Theorem~\ref{TTalagrand} is sharp (up to the value of $c$).
This is also demonstrated by the examples of Section~\ref{sec:sub:most-general}.
Still, one may wonder whether it can be improved
when one of the two
sets is antipodal.


\subsection{The inequality of Keller, Mossel and Sen}
\label{sec:sub:KMS}

The following relative of Theorem~\ref{TTalagrand}
is from~\cite{KMS14}.
\begin{theorem}\label{TKMS}
There is a fixed $c>0$ such that for increasing $\A ,\B  \sub \Omega $,
\beq{EKMS}
\mbox{$\mathrm{Cor}(\A ,\B ) \geq c \sum \psi(I_k(\A )) \psi(I_k(\B )),$}
\enq
where $\psi(x)=x/\sqrt{\log(e/x)}$.
\label{Thm:KMS14}
\end{theorem}
\nin
Like Theorem~\ref{TTalagrand}, this gives
a weak version of Chv\'{a}tal's conjecture; here we replace \eqref{THarper} by a theorem
of Talagrand~\cite{Talagrand94} that sharpens the KKL Theorem:
\begin{theorem}\label{TTalagrand2}
For $\B \subset \Omega$ increasing, $\sum \varphi(I_k(\B )) > c \mu(\B) (1-\mu(\B))$ (where
$c$ is a positive constant).
\end{theorem}

\begin{corollary}\label{Claim:Half-Weak}
There is a fixed $c>0$ such that for any increasing $\A,\B  \sub \Omega $,
\[
\mathrm{Cor}(\A ,\B ) > c \psi(\Imin(\A )) \mu(\B)(1-\mu(\B)).
\]
\end{corollary}

\begin{proof}
Theorem~\ref{Thm:KMS14},
the monotonicity of $\psi$ and Theorem~\ref{TTalagrand2}
give
\beq{CKMS}
\mbox{$\mathrm{Cor}(\A ,\B ) >
c\psi(\Imin(\A ))  \sum \psi(I_k(\B ))$}  > c \psi(\Imin(\A ))\mu(\B)(1-\mu(\B))
\enq
(where the second inequality uses the fact that $\psi(x) \geq \varphi(x)$ for $x\in [0,1]$).
\end{proof}

Corollary~\ref{Claim:Half-Weak} misses the bound of Conjecture~\ref{Conj:Equiv-Chvatal} by a
factor like $\sqrt{\log(1/ \Imin(\A ))}$,
which improves the $\log(1/ \Imin(\A ))$
of Corollary~\ref{Claim:Weak} but is still  weak.
Of course something is lost in the second inequality of \eqref{CKMS},
but we don't see how to exploit this in general. (For regular $\B$,
Theorem~\ref{Thm:KMS14} does support a
different derivation of Proposition~\ref{Prop:Weakly-Symmetric}.)
It is tempting to try to replace the bound in \eqref{EKMS} by
$
c \sum \psi_\alpha(I_k(\A )) \psi_{1-\alpha}(I_k(\B )),
$
where $\psi_\alpha(x)=x/(\log(e/x))^{\alpha}$
(e.g. $\psi_0$ is the identity, $\psi_{1/2}=\psi$
and $\psi_1=\varphi$).
If true for $\ga =0$, this would give
Conjecture~\ref{Conj:Equiv-Chvatal} to within
a constant factor {\em via} the argument of Corollary~\ref{Claim:Half-Weak}
(since it replaces the middle expression in \eqref{CKMS} by
$
c \Imin(\A ) \sum\varphi(I_k(\B )) $); but in fact it is not true for any
$\alpha \neq 1/2$ (e.g. for $\alpha <1/2$ let $\B$ be ``majority" and
$\A = \{x \in \Omega : \sum_i x_i > s\}$, with $s$ chosen so that $\mu(\A) = \exp[-\gO(n)]$).

\section{Alternative correlation inequalities}
\label{sec:correlation2}

Here we consider a few alternative correlation inequalities.
Some of these (if correct)
would imply Chv\'{a}tal's conjecture, while others
may serve as first steps in the direction of the conjecture.
Proofs are given in Section~\ref{sec:proofs}.



\subsection {Reformulations and consequences of Kahn's conjecture}

\mn
We begin with a pair of inequalities that
reformulate
Conjecture~\ref{Conj:Kahn-Equivalent} and the special case of
Conjecture~\ref{Conj:Kahn-Convex}
suggested in \eqref{average}
for $f$ of the form $2 \cdot 1_{\mathcal{\B}}-1$
(equivalently, for $\{\pm 1\}$-valued $f$).
Recall $\hat \C=\hat\chi_{_\C}$.

\begin{conjecture}\label{Conj:Kahn-Correlation}
For increasing $\A  \sub \Omega $ and maximal
intersecting $\B  \sub \Omega $,

\mn
{\rm (a)}
$ ~~\mathrm{Cor}(\A ,\B ) \geq \sum_{i=1}^n I_i(\A ) \sum\{
\hat \B (S)^2:\max(S)=i\}$,

\mn
{\rm (b)}
$ ~~
\mathrm{Cor}(\A ,\B ) \geq \sum_{i=1}^n I_i(\A ) \sum\{
\frac {1} {|S|}\hat \B (S)^2:S\ni i\}.
$

\end{conjecture}

\nin
Each of these inequalities has the form $\mathrm{Cor}(\A ,\B ) \geq \sum  w_i I_i(\A )$,
with
$\sum w_i = \sum \{ \hat B(S)^2: S \ne \emptyset\}$ $= \mu (\B) (1-\mu(\B))=1/4$;
thus either
implies $\mathrm{Cor}(\A ,\B ) \geq \frac {1}{4} \Imin (\A)$.
For comparison with Section~\ref{sec:correlation},
note that in each case, $w_i\leq \sum\{\hat \B (S)^2:S\ni i\}= \frac 14 I_i(\B)$.
It is easy to see that (a) is strongest when we index with $I_1(\A)\leq \cdots \leq I_n(\A)$,
and that (a) implies (b) (by averaging over orderings).
On the other hand, for
Chv\'atal's conjecture it is enough to establish the \emph{weakest} version of (a),
in which we take $I_1(\A)\geq \cdots \geq I_n(\A)$.

\mn

The following generalization of Chv\'{a}tal's conjecture avoids the antipodality condition.
\begin{conjecture}\label{Statement:Sum-with-Dual}
For increasing $\A ,\B  \sub \Omega $,
\[
\mathrm{Cor}(\A ,\B ) + \mathrm{Cor}(\A ,\B ') \geq 2\mu(\B)(1-\mu(\B)) \Imin(\A ).
\]
\end{conjecture}
\nin
Note that if $\B $ is antipodal, then $\B '=\B $ and $\mu(\B)(1-\mu(\B))=1/4$; thus
Conjecture~\ref{Statement:Sum-with-Dual} contains Conjecture~\ref{Conj:Equiv-Chvatal}.

\begin{proposition}
\label {Prop:reduction}
If Conjecture~\ref{Conj:Kahn-Equivalent} is true then for increasing $\A,\B  \sub \Omega $,
\begin{equation}\label{Eq:Reduction0}
\mathrm{Cor}(\A ,\B ) \geq \frac{1}{2}\sum I_i(\A ) \sum\{
\hat \B (S)^2:\max(S)=i\}.
\end{equation}
\end{proposition}
\nin
Inequality \eqref{Eq:Reduction0} implies
Conjecture~\ref{Conjecture:Balanced} with $c=1/2$
(a consequence of Conjecture~\ref{Conj:Kahn} mentioned in Section~\ref{Intro}), so also
Conjecture~\ref{Statement:Sum-with-Dual} without the 2.

\subsection {A symmetric version of Conjecture~\ref{Conj:Kahn-Correlation}(b)}

For $i \in [n]$, define $\Delta_i: \mathbb{R}^{\Omega} \rightarrow \mathbb{R}^{\Omega}$ by
$\Delta_i f(x) = f(x)-f(x \oplus e_i)$. It is easy to see that for any $S \sub [n]$,
$\widehat{\Delta_i f}(S) = \chi_{\{i \in S\}} \hat f(S)$. We will use $\Delta_i(\A)$ for $\Delta_i(1_{\A})$.

For $g:\Omega  \rightarrow \mathbb{R}$ with $\mathbb{E}[g]=0$ and $\alpha \in [0,1]$, set
$M_\alpha(g) = \sum_{S} \frac{\hat f(S)^2}{|S|^{\alpha}}.$

\begin{conjecture}\label{Conj:Kahn-Correlation-alpha}
There is a fixed $c>0$ such that for any
$\A ,\B \sub \Omega $ with $\A$ increasing and $\B  $ increasing
and balanced,
\begin{equation}\label{Eq:Kahn-Non-diagonal}
\mathrm{Cor(\A ,\B )} \geq c \sum_{i} M_{\alpha} (\Delta_i(\A)) M_{1-\alpha} (\Delta_i(\B))
= c \sum_{S,T \neq \emptyset} \frac{|S \cap T| \hat \A (S)^2 \hat \B (T)^2}{|S|^{\alpha}|T|^{1-\alpha}}.
\end{equation}
\end{conjecture}
\nin
(Using $I_i(\A)=4\sum\{\hat \A^2(T):T\ni i\}$, the sum
in Conjecture~\ref{Conj:Kahn-Correlation}(b) may be rewritten as
$\sum_{S,T} |S \cap T| \hat \A (S)^2 \hat \B (T)^2/|T|$.)
We can at least prove the symmetric version of Conjecture~\ref{Conj:Kahn-Correlation-alpha}:

\begin{proposition}\label{Prop:Kahn-Correlation-Weaker}
There is a fixed $c>0$ such that for any
$\A ,\B \sub \Omega $ with $\A$ increasing and $\B  $ increasing
and balanced,
\begin{equation}\label{Eq:Kahn-symmetric}
\mathrm{Cor(\A ,\B )} \geq c \sum_{i} M_{1/2} (\Delta_i(\A)) M_{1/2} (\Delta_i(\B))
= c \sum_{S,T \neq \emptyset} \frac{|S \cap T| \hat \A (S)^2 \hat \B (T)^2}{\sqrt{|S||T|}}.
\end{equation}
\end{proposition}

\subsection {Diagonal forms of Conjecture~\ref{Conj:Kahn-Correlation}(b)}

An immediate consequence of Proposition \ref{Prop:Kahn-Correlation-Weaker} is

\begin{corollary}\label{Cor:Kahn-Correlation-Weaker}
There is a fixed $c>0$ such that for any 
increasing $\A ,\B  \sub \Omega $,
\begin {equation}
\label{Eq:Kahn-Correlation-Weaker}
\mbox{$\mathrm{Cor(\A ,\B )} \geq c \sum_{S \neq \emptyset} \hat \A (S)^2 \hat \B (S)^2.$}
\end {equation}
\end{corollary}

\mn
{\em Remark.}
The inequality (\ref {Eq:Kahn-Correlation-Weaker}) is a lower bound on the correlation of two
increasing functions in terms of
the (normalized) $\ell_2$-norm of their convolution. It would be interesting to
extend it to other contexts and to find a proof that's more direct than the one in
Section~\ref{PP5.5}.
A (very) weak consequence of
Conjecture~\ref{Conj:Kahn-Correlation} (see the sentence following
Conjecture~\ref{Conj:Kahn-Correlation-alpha}) is:

\begin{conjecture}\label{Statement:Kahn-Correlation-Weaker}
For increasing $\A  \sub \Omega $ and maximal
intersecting $\B  \sub \Omega $,
\begin{equation}\label{Eq:Kahn-Correlation-Weaker2}
\mbox{$\mathrm{Cor(\A ,\B )}
\geq 4\sum_{S \neq \emptyset} \hat \A (S)^2 \hat \B (S)^2.$}
\end{equation}
\end{conjecture}
\nin
We expect even more to be true:

\begin{conjecture}\label{Statement:stronger-diagonal}
For any increasing $\A ,\B  \sub \Omega $,
\[
\mbox{$\mathrm{Cor(\A ,\B )} \geq  \sum_{S \neq \emptyset} |S| \hat \A (S)^2 \hat \B (S)^2.$}
\]
\end{conjecture}


\subsection{Inequalities involving the total influence}

We would like to (perhaps optimistically) suggest the following family of inequalities.

\begin{conjecture}\label{Statement:Gil's-inequality}
There is a fixed $c>0$ such that for increasing $\A ,\B  \sub \Omega $ and $\alpha \in [0,1]$,
\begin{equation}\label{Eq:Gil's-inequality}
\mathrm{Cor}(\A ,\B ) \geq c \left(\tfrac{\mathrm{Var} (\A) }{I(\A )} \right)^{\alpha}
\left(\tfrac{\mathrm{Var} (\B) }{I(\B )} \right)^{1-\alpha}\sum  I_i(\A ) I_i(\B )
\end{equation}
\end{conjecture}
\nin
(where $\mathrm{Var} (\A)=\mu(\A)(1-\mu(\A)$).
The case $\alpha=0$ would imply Conjecture~\ref{Conjecture:Balanced},
giving, for increasing $\A $ and $\B $,
\[
\mathrm{Cor}(\A ,\B ) \geq c \mathrm{Var} (\B) \sum  I_i(\A )\tfrac{ I_i(\B ) }{I(\B )} \geq
c\mathrm{Var} (\B) \Imin (\A ) \sum \tfrac{ I_i(\B ) }{I(\B )} = c\mu(\B)(1-\mu(\B))\Imin (\A ).
\]
For $\ga=1/2$,
Conjecture~\ref{Statement:Gil's-inequality} is reminiscent of Theorem~\ref{Thm:KMS14}, and
we are inclined to believe that it is true, at least in this case.

We may also strengthen Conjecture~\ref{Statement:Sum-with-Dual} to
a variant of the case $\ga=0$ of Conjecture~\ref{Statement:Gil's-inequality}:
\begin{conjecture}\label{Statement:Gil's-inequality-dual}
For increasing $\A ,\B  \sub \Omega $,
\begin{equation}
\mbox{$\mathrm{Cor}(\A ,\B ) + \mathrm{Cor}(\A ,\B ') \geq 2 \mu(\A)(1-\mu(\A))
\sum  I_i(\A ) I_i(\B )/I(\B ).$}
\end{equation}
\end{conjecture}



\section{Alternative correlation inequalities: Proofs}
\label{sec:proofs}

\subsection {Connection with Conjectures~\ref{Conj:Kahn-Equivalent} and \ref{Conj:Kahn-Convex} }
\label{Connection}

Here we show equivalence of Conjecture~\ref{Conj:Kahn-Correlation}(a)
and the restriction of Conjecture~\ref{Conj:Kahn-Equivalent} to $f$'s of the form
$2\chi_{_\f}-1$ with $\f\sub\gO$ antipodal (maximal intersecting).
A similar argument shows that (b)
is equivalent to Conjecture~\ref{Conj:Kahn-Convex} for the same class of $f$'s
and $\gl_S$'s as in \eqref{average}.

\begin{proof}
For $f$ as above
we have $f^* = \chi_{_\f}$, so the inequality
\eqref{AinmathcalI} of
Conjecture~\ref{Conj:Kahn-Equivalent} becomes
\[ 
\mbox{$\sum_{A \in \mathcal{I}} \tilde{f}(A)  \geq \sum_{A \in \mathcal{I}} \chi_{_\f}(A)
= |\mathcal{F} \cap \mathcal{I}|, $}
\] 
which we may rewrite as
\begin{equation}\label{Eq:Kahn-Correlation2}
\langle \tilde{f}, \chi_{_{\mathcal{I}}} \rangle \geq \mu(\mathcal{F} \cap \mathcal{I})
\end{equation}
(since
$
\sum_{A \in \mathcal{I}} \tilde{f}(A)  = \sum_A \tilde{f}(A) \chi_{_{\mathcal{I}}}(A) =
2^n \langle \tilde{f}, \chi_{_{\mathcal{I}}} \rangle$).
%
Note also that
\[
\mbox{$\mu(\tilde{f}) =\sum \lambda_i(f) \mu(\chi_{_i}) =1/2 ~~(=\mu(\f))$}
\]
(using $\mu(\cdot)$ for expectation), since
the $\gl_i$'s are convex coefficients
(see following Conjecture~\ref{Conj:Kahn}).
Thus \eqref{Eq:Kahn-Correlation2} is equivalent to
$
\mathrm{Cor} (\tilde{f}, \chi_{\mathcal{I}}) \geq \mathrm{Cor}(\mathcal{F},\mathcal{I})
$
or, with $\J=\I^c$ (see paragraph following Proposition~\ref{Prop:Equivalence}),
\begin{equation}\label{Eq:Kahn-Correlation5}
\mathrm{Cor}(\mathcal{F},\mathcal{J})\geq \mathrm{Cor} (\tilde{f}, \chi_{_\J}) .
\end{equation}
To evaluate the r.h.s. notice that,
with $\gl_i(f)=\gl_i$,
\[
\mbox{$\tilde{f}=\sum \gl_i\chi_{_i} = \sum \gl_i(1-u_{\{i\}})/2
= 1/2 -(1/2)\sum\gl_iu_{\{i\}}$}
\]
(note $\chi_{_i} = (1-u_{\{i\}})/2$)---that is, the
Fourier coefficients, $\ga_S$, of $\tilde{f}$ are given by:
$\ga_\0=1/2$; $\ga_{\{i\}}= -\gl_i/2$;
and $\ga_S=0$ if $|S|>1$---and that
for increasing $\J\sub \gO$,
$
\hat\chi_{_\J}(\{i\}) = - I_i(\mathcal{J})/2
$
(for any $i$).
Thus, recalling \eqref{Parseval2}, we have
\[
\mbox{$\Cor(\tilde{f},\chi_{_\J}) = -\tfrac{1}{2}\sum \gl_i(-\tfrac{1}{2}I_i(\J))
= \tfrac{1}{4} \sum I_i(\J)\sum\{\hat f (S)^2:\max(S)=i\}$;}
\]
so \eqref{Eq:Kahn-Correlation5}
is the same as Conjecture~\ref{Conj:Kahn-Correlation}(a)
(with $(\A,\B)=(\J,\f)$).
\end{proof}

\subsection {Proof of Proposition \ref {Prop:reduction}}

Regard
$\A$ and $\B$ as subsets of $2^{[2,n]}$
(with $[2,n]= \{2,3,\ldots,n\}$), define $\A',\B' \sub 2^{[n]}$
by \[
\mbox{$\A' = \A \cup \{A \cup \{1\}:A \in \A\}~$
and $~\B' = \{B \cup \{1\}:B \in \B\}$,}
\]
and set $\I=(\A')^c$.
Let $f:\{0,1\}^{n+1} \to \{-1,0,1\}$ be the antipodal function with
$f(x) \equiv 1$ on $\B'$ and
$f(x)\equiv 0$ on $\{B: 1\in B \not \in \B'\}$.
The argument of Section~\ref{Connection} may be repeated essentially
\emph{verbatim} to
show that the inequality
$\sum_{A \in \mathcal{I}} \tilde{f}(A) \geq \sum_{A \in \mathcal{I}} f^*(A)$ of
Conjecture~\ref{Conj:Kahn-Equivalent}
is equivalent to
\beq{equivineq}
\mbox{$\mathrm{Cor}(\A' ,\B') \geq \frac{1}{4}\sum_i I_i(\A') \sum\{\hat f (S)^2:\max(S)=i\}$.}
\enq
This implies
Proposition~\ref{Prop:reduction}
as follows.

Writing $\mu$ and $\mu'$ for uniform measure on $2^{[2,n]}$ and $2^{[n]}$ respectively,
we easily see, first, that
$\mu'(\A')=\mu(\A)$, $\mu'(\B')=\mu(\B)/2$ and $\mu'(\A'\cap \B')=\mu(\A\cap \B)/2$,
implying
\[
\Cor(\A',\B') = \Cor(\A,\B)/2,
\]
and, second, that $I_1(\A')=0$ and $I_i(\A')=I_i(\A)$ for $i\in [2,n]$.
Moreover it is easy to see
that for $S\sub [2,n]$,
\[
\hat \B(S) =\left\{\begin{array}{ll}
\hat f(S)&\mbox{if $|S|$ is even,}\\
-\hat f(S\cup\{1\})&\mbox{if $|S|$ is odd,}
\end{array}\right.
\]
which accounts for all of $\hat f$ since antipodality
implies $\hat f(T)=0$ if $T$ is even.
Finally, combining these observations, we find that
\eqref{equivineq} is in fact the same as \eqref{Eq:Reduction0}.\qed

\subsection {Proof of Proposition~\ref{Prop:Kahn-Correlation-Weaker}}
\label{PP5.5}

We need the following extension of Talagrand's~\cite[Prop. 2.3]{Talagrand94}.
\begin{lemma}\label{Lemma:Kahn-Correlation-Weaker}
For any $\alpha \in [0,1]$ there is a $c=c(\alpha)$ such that
for any $f:\Omega  \rightarrow \mathbb{R}$ with $\mathbb{E}[f]=0$,
\[
M_\alpha(f) := \sum_{S} \frac{\hat f(S)^2}{|S|^{\alpha}} \leq c \left(\log \left(\frac{e\|f\|_2}{\|f\|_1} \right) \right)^{-\alpha}
\|f\|_2^2.
\]
\end{lemma}

Talagrand proves this with $\alpha=1$ but for more general product measures $\mu_p$.
(Proposition~\ref{Prop:Kahn-Correlation-Weaker} below also holds in this greater generality, given natural
definitions which we omit.) At any rate, the proof of Lemma~\ref{Lemma:Kahn-Correlation-Weaker} follows
his nearly {\em verbatim} and will not be given here.


\begin{proof}[Proof of Proposition~\ref{Prop:Kahn-Correlation-Weaker}.]
For any $\C \sub 2^{[n]}$, $f=\Delta_i(\C)$ satisfies $f(x) \in \{0,\pm 1\}$
for all $x$, $\mathbb{E}[f]=0$,
and $||f||_2^2 = ||f||_1 = I_i(\C)$. Thus, Lemma~\ref{Lemma:Kahn-Correlation-Weaker} gives
(with $\psi$ as in Section~\ref{sec:sub:KMS})
\[
M_{1/2}(f) \leq c' I_i(\C) \left( \log (e/\sqrt{I_i(\C)}) \right)^{-1/2} \leq c'\psi(I_i(\C)).
\]
Applying this for each $i \in [n]$
and $\C \in \{\A, \B\}$ and using Theorem~\ref{Thm:KMS14}, we have
\[
\sum M_{1/2} (\Delta_i(\A)) M_{1/2} (\Delta_i(\B)) \leq
(c')^2 \sum \psi(I_i(\A)) \psi(I_i(\B)) \leq c'' \mathrm{Cor}(\A ,\B ),
\]
completing the proof.\end{proof}


\section{Can the antipodality assumption be removed?}
\label{sec:sub:most-general}

Here we show that, as mentioned earlier, Conjecture~\ref{Conjecture:Balanced}
fails for $c > \ln 2$, even assuming $\B$ is balanced; in particular we cannot
relax the antipodality
in Conjecture~\ref{Conj:Equiv-Chvatal}
to the requirement that $\B$ be balanced and increasing.
%
Our example is based
on the ``tribes" construction of Ben-Or and Linial~\cite{BL}.
(For simplicity we settle for $\B$ only {\em approximately} balanced.)

\mn \textbf{Example.} To define the tribes family $\A  $ we consider an equipartition $[n]=S_1\cup\cdots \cup S_{n/r}$
(with $r$ to be specified; for present purposes we assume $r|n$), and, now thinking of $\gO$ as $2^{[n]}$,
set
\[
\A =\{A\sub [n]:\exists i, A\supseteq S_i\}.
\]
Then
$\B:=\A' =\{B\sub [n]: B\cap S_i\neq\0 ~\forall i\}.$
These are of course increasing with $\mu(\B)=1-\mu(\A)$ (as for any dual pair).
To arrange $\mu(\A)\sim 1/2$ we take $r=\lfloor r(n)\rfloor$, where
\[
r(n) = \log_2 n - \log_2 \log_2 n + \log_2 (\log_2e) ,
\]
for simplicity
confining ourselves to $n$'s for which $r|n$ and $r(n)-r=o(1)$.
We then have
\[
\mu(\B) = (1-2^{-r})^{n/r} = \exp[- 2^{-r}n/r +O(4^{-r}n)]
\ra 1/2
\]
(since $2^r\sim n/(r\ln 2)$, where as usual $a\sim b$ means $a/b\ra 1$).

For the correlation we work with $\A^c$; we have (with a little calculation)
\[
\mu(\B|\A^c) = (1- (2^r-1)^{-1})^{n/r} = \mu(\B) \left[\tfrac{2^r(2^r-2)}{(2^r-1)^2}\right]^{n/r}
= \mu(\B)\left(1- (2^r-1)^{-2}\right)^{n/r},
\]
whence
\[
\Cor(\A,\B) = -\Cor(\A^c,\B)
= -\mu(\A^c)\mu(\B)\left[\left(1- (2^r-1)^{-2}\right)^{n/r} -1\right]
\sim \mu(\A^c)\mu(\B) (r4^r)^{-1}n.
\]
On the other hand the influence of $i\in S_j$ on $\A$ is the probability that a
uniform subset of $[n]$ contains $S_j\sm\{i\}$ and contains no $S_\ell$
with $\ell\neq j$; the common value of the $I_k(\A)$'s is thus
\[
2^{-r+1} (1-2^{-r})^{(n/r)-1} \sim 2^{-r+1} \mu(\B)
\]
and we have
\[
\Cor(\A,\B)/\Imin(\A)
\sim 2^{r-1}\mu(\A^c) (r4^r)^{-1}n
\sim \frac{n}{4r2^r}
\sim \frac{\ln 2}{4}~.
\]
\qed

It is perhaps surprising (or suggestive?) that the above $\B$'s are so
different from the families $\f_i$ that provide the
lower bound in Conjecture~\ref{Conj:Equiv-Chvatal}.

\section{Chv\'{a}tal's Conjecture is true ``on average"}
\label{sec:average}

Another initially plausible
inequality suggested by \eqref{THarper}, is
\begin{equation}\label{Eq:Wrong3}
\mathrm{Cor}(\A ,\B ) \geq \tfrac12 \mu(\B ) \log_2(1/\mu(\B ))
\Imin(\A ) ;
\end{equation}
The example from the previous section shows that this is false even when $\A,\B$ are balanced.
When they are not balanced
the failure of \eqref{Eq:Wrong3} is more severe; e.g.  when
$\A = 2^{[n]}\sm \{\0\}$, we have
$\mathrm{Cor}(\A ,\B ) =\tfrac12 \mu(\B ) \Imin(\A ) $ for any $\B\neq 2^{[n]}$.)
But as we will see in this section,~(\ref{Eq:Wrong3}) does hold
{\em on average} when $\A$ and $\B$ are drawn from a family
of increasing sets of equal measure.





In~\cite{Keller09}, the fourth author proved the following ``average-case'' variant of
Theorem~\ref{TTalagrand}.
\begin{theorem}\label{Thm:Average-Talagrand}
For a family $\f$ of increasing subsets of $\gO$
and $\A,\B$ drawn uniformly and independently from $\f$,
\[
\mbox{$\mathbb{E}\mathrm{Cor}(\A ,\B ) \geq \tfrac14 \mathbb{E}
\sum I_k(\A ) I_k(\B ) .$}
\]
\end{theorem}
\nin
This immediately implies a corresponding variant
of Chv\'{a}tal's Conjecture, even in the more general setting of Conjecture~\ref{Conjecture:Balanced}:
\begin{proposition}\label{Prop:Average-Chvatal}
For a family $\f$ of increasing subsets of $\gO$, each of measure $t\in (0,1)$,
and $\A,\B$ drawn uniformly and independently from $\f$,
\[
\mbox{$\mathbb{E}\mathrm{Cor}(\A ,\B )
\geq \tfrac12 t \log_2 (1/t)\mathbb{E}
\Imin(\A ) .$}
\]
In particular when $t=1/2$,
\[
\mbox{$\mathbb{E}\mathrm{Cor}(\A ,\B )
\geq \tfrac14\mathbb{E}
\Imin(\A ) .$}
\]
\end{proposition}
\mn
Thus Conjecture~\ref{Conj:Equiv-Chvatal} does hold in an average sense;
but note that this is true even with balance in place of antipodality, where we have seen that
the conjecture proper is not true.  More generally, Proposition~\ref{Prop:Average-Chvatal}
gives truth on average of the incorrect inequality \eqref{Eq:Wrong3}.

\begin{proof}
Using Theorem~\ref{Thm:Average-Talagrand} for the first inequality and
\eqref{THarper} for the last, we have
\[
\mbox{$\mathbb{E}\mathrm{Cor}(\A ,\B )
~\geq~
\tfrac14 \mathbb{E} \sum I_k(\A ) I_k(\B )
~\geq~
\tfrac14 \mathbb{E} [\Imin(\A) I(\B )]
~\geq ~
\tfrac12 t \log_2 (1/t)\mathbb{E}
\Imin(\A ) .$}
\]
\end{proof}

We next show that Proposition~\ref{Prop:Average-Chvatal} can sometimes be strengthened.
Here we need another result of
Talagrand~\cite{Talagrand96} and Chang~\cite{Chang} 
(see also~\cite{IMR14} for the constant):
\begin{theorem}\label{Thm:Chang's-lemma}
For increasing $\B  \sub \Omega $,
$~
\sum I_k(\B )^2 \leq 8\mu(\B)^2 \ln(1/\mu(\B)).
$
\end{theorem}

\mn
For $\A\sub \gO$ and $\gc>0$, write $s_\gc(\A)$ for the sum of the smallest
$\gc\log_2(1/\mu(\A))$ influences of $\A $.
(So we are now thinking of $\mu(\A)$ as somewhat small.
Of course, strictly speaking, we should say $\gc\log_2(1/\mu(\A))\in \mathbb{N}$.)
\begin{proposition}\label{Prop:Average-Chvatal-Small-Improved}
For a family $\f$ of increasing subsets of $\gO$, each of measure $t\in (0,1)$,
and $\A,\B$ drawn uniformly and independently from $\f$,
\[
\mbox{$\mathbb{E}\mathrm{Cor}(\A ,\B )
\geq (4\gc)^{-1}(2-2\sqrt{2\gc \log_2e}~) t \cdot\mathbb{E} s_\gc(\A ) .$}
\]
\end{proposition}

\begin{proof}
By Theorem~\ref{Thm:Average-Talagrand}
it suffices to show that for increasing $\A ,\B  \sub \Omega $
with $\mu(\A )=\mu(\B )=t$,
\begin{equation}\label{Eq:Average1}
\mbox{$\sum  I_k(\A ) I_k(\B ) \geq  \gc^{-1}(2-2\sqrt{2\gc \log_2e}~)t \cdot s_\gc(\A );$}
\end{equation}
this is seen as follows.
We may assume that $I_1(\A ) \leq \cdots \leq I_n(\A )$ and then,
since $\B$ appears only on the l.h.s. of \eqref{Eq:Average1}, that
$I_1(\B ) \geq \cdots \geq I_n(\B )$
(by the ``Rearrangement Inequality,"
e.g. \cite[Theorem~368]{HLP})).

Set $q= \gamma \log_2(1/t)$.
Using Theorem~\ref{Thm:Chang's-lemma} and Cauchy-Schwarz we have
\[
\mbox{$8t^2\ln(1/t) \geq \sum_{k\leq q}I_k(\B)^2 \geq (1/q)\left(\sum_{k\leq q}I_k(\B)\right)^2,$}
\]
implying
$
\mbox{$\sum_{k\leq q}I_k(\B) \leq t\sqrt{8q\ln(1/t)} = \ga t\log_2(1/t)$},
$
with $\ga = 2\sqrt{2\gamma \log_2e}$,
and, by \eqref{THarper},
\[
\mbox{$\sum_{k> q}I_k(\B) \geq  (2-\alpha) t\log_2(1/t)$.}
\]
But then
\[
\mbox{$\sum I_k(\A)I_k(\B) ~\geq~ I_q(\A)\sum_{k>q}I_k(\B)
~\geq~
(s_\gc(\A)/q)(2-\alpha) t \log_2(1/t)
~=~ \gc^{-1}(2-\ga)t \cdot s_\gc(\A ) .$}
\]
\end{proof}


\begin{thebibliography}{99}


\bibitem{Ahlswede-Daykin} R. Ahlswede and D. E. Daykin,
An inequality for the weights of two families of sets, their unions and intersections,
{\em Probab. Th. Rel. Fields} {\bf 43(3)}, pp.~183--185, 1978.

\bibitem{BL} M. Ben-Or and N. Linial, Collective coin flipping, in {\it Randomness and Computation}
(S. Micali, ed.), Academic Press, New York, 1990, pp.~91--115.


\bibitem{Chang} M.-C. Chang, A polynomial bound in Freiman's theorem, {\it Duke Math. J.}
\textbf{113(3)}, pp.~399--419, 2002.

\bibitem{Chvatal} V. Chv\'{a}tal, Intersecting families of edges in hypergraphs having the hereditary property.
Hypergraph Seminar (Proc. First Working Sem., Ohio State Univ., Columbus, Ohio, 1972; dedicated to Arnold Ross),
pp. 61--66. Lecture Notes in Math., Vol. 411, Springer, Berlin, 1974.

\bibitem{Eberhard} S. Eberhard, Kahn-Kalai-Linial for Intersecting Upsets, MathOverflow, Question no. 105086,
2012.

\bibitem{EK74} P. Erd\H{o}s and D. J. Kleitman, Extremal problems among subsets of a set, {\it Disc.
Math.}, \textbf{8}, pp.~281--294, 1974.

\bibitem{EKR} P. Erd\H{o}s, C. Ko, and R. Rado, Intersection theorems for systems of finite sets, {\it The
Quarterly Journal of Mathematics, Oxford, Second Series}, \textbf{12}, pp.~313--320, 1961.

\bibitem{Fishburn} P. C. Fishburn, Combinatorial optimization problems for systems of subsets,
{\it SIAM Review}, \textbf{30}, pp.~578--588, 1988.

\bibitem{FKG} C. M. Fortuin, P. W. Kasteleyn, and J. Ginibre, Correlation inequalities on some partially
ordered sets, {\it Comm. Math. Phys.} \textbf{22}, pp.~89--103, 1971.

\bibitem{HLP} G. H. Hardy, J. E. Littlewood, and G. P\'{o}lya, Inequalities, 2nd Edition, Cambridge
University Press, 1952.

\bibitem{Harper} L. H. Harper, Optimal assignment of numbers to vertices, {\it J. Soc. Ind. Appl. Math.},
\textbf{12}, pp.~131--135, 1964.

\bibitem{Harris} T. E. Harris, A lower bound for the critical probability in a certain percolation process,
{\it Proc. Cambridge Phil. Soc.} \textbf{56}, pp.~13--20, 1960.

\bibitem{IMR14} R. Impagliazzo, C. Moore, and A. Russell, An entropic proof of Chang's inequality,
{\it SIAM J. Disc. Math.} \textbf{28(1)}, pp.~173--176, 2014.

\bibitem{Kahn-unpublished} J. Kahn, A conjecture implying Chv\'{a}tal's conjecture, unpublished manuscript,
circa 1990.

\bibitem{KKL} J. Kahn, G. Kalai, and N. Linial, The influence of variables on Boolean functions, Proc. 29-th
Ann. Symp. on Foundations of Comp. Sci., pp.~68--80, Computer Society Press, 1988.


\bibitem{Keller09} N. Keller, Lower bound on the correlation between monotone families in the average
case, {\it Adv. Appl. Math.} \textbf{43}, pp.~31--45, 2009.


\bibitem{KMS14} N. Keller, E. Mossel, and A. Sen, Geometric influences II: Correlation inequalities and
noise sensitivity, {\it Ann. Inst. Henri Poincare} \textbf{50(4)}, pp.~1121--1139, 2014.

\bibitem{Kleitman} D. J. Kleitman, Extremal hypergraph problems, in: Proceedings of the 7th British
Combinatorial Conference (B. Bollob\'{a}s, ed.), pp.~44--65, Cambridge University Press, 1979.

\bibitem{Talagrand94} M. Talagrand, On Russo's approximate 0-1 law, {\it Ann. Probab.} \textbf{22},
pp.~1576--1587, 1994.

\bibitem{Talagrand96} M. Talagrand, How much are increasing sets positively correlated?, {\it Combinatorica}
\textbf{16(2)}, pp.~243--258, 1996.

\end{thebibliography}
\end {document}